\documentclass[11pt,twoside]{amsart}
\usepackage{amsmath,amsfonts,amssymb,amsthm,mathrsfs} 
\usepackage[matrix,arrow]{xy}

\usepackage{verbatim}
\usepackage[latin1]{inputenc}   
\usepackage{url}

\title[Explicit projective embeddings of standard opens]{Explicit projective embeddings of standard opens of the Hilbert scheme of points}
\author[R.\ M.\ Skjelnes \and G.\ S.\ St{\aa}hl]{Roy Mikael Skjelnes \and Gustav S{\ae}d{\'e}n St{\aa}hl}
\thanks{The second author is supported by the Swedish Research Council, grant number 2011-5599.}
\address{Department of Mathematics, KTH Royal Institute of Technology, SE-100 44 Stockholm, Sweden}
\email{skjelnes@kth.se\\gss@math.kth.se}


\DeclareMathOperator{\Spec}{Spec}

\newcommand{\Q}[1]{{E}^{{#1}}}
\newcommand{\Hbeta}{\operatorname{Hilb}^{\beta}}
\newcommand{\Gr}[1]{D_{+}({#1})}
\newcommand{\xbeta}[1]{z^{{#1}}\beta}
\newcommand{\Fitt}{\operatorname{Fitt}}
\newcommand{\VPS}[2]{{\operatorname{VPS}^{{#1}}_{{#2}}}}

\newcommand{\ra}{\longrightarrow}
\newcommand{\Proj}{\mathbb{P}}

\newcommand{\A}{\mathbb{A}}

\newcommand{\eqbeg}{\begin{equation}}
\newcommand{\eqend}{\end{equation}}

\newcommand{\Sgot}{\mathfrak{S}}

\newcommand{\m}{\lambda}

\newcommand{\calI}{\mathscr{I}}

\newcommand{\calB}{\mathscr{B}}

\newcommand{\calO}{\mathscr{O}}

\newcommand{\Snbimap}{\xymatrix@M=1pt{ {\Sgot_n\times U}
\ar@<.5ex>[r] \ar@<-.5ex>[r] & {U}}}
\newcommand{\bimap}{\xymatrix@M=1pt{ {R}
\ar@<.5ex>[r] \ar@<-.5ex>[r] & {X}}}
\newcommand{\pbbimap}{\xymatrix@M=1pt{ {R\times_{A}Y}
\ar@<.5ex>[r] \ar@<-.5ex>[r] & {X\times_AY}}}
\newcommand{\grpbimap}{\xymatrix@M=1pt{ {G\times X}
\ar@<.5ex>[r] \ar@<-.5ex>[r] & {X}}}
\newcommand{\pibimap}{\xymatrix@M=1pt{ {R}
\ar@<.5ex>[r]^-{\pi_1} \ar@<-.5ex>[r]_-{\pi_2} & {X}}}
\newcommand{\quotmap}{\xymatrix@M=1pt{ {R/G}
\ar@<.5ex>[r]^-{\pi_1} \ar@<-.5ex>[r]_-{\pi_2} & {U/G}}}

\newtheorem{summary}{Theorem}
\newtheorem{thm}[subsection]{Theorem}
\newtheorem{lemma}[subsection]{Lemma}
\newtheorem{cor}[subsection]{Corollary}
\newtheorem{prop}[subsection]{Proposition}

\theoremstyle{definition}
\newtheorem{defn}[subsection]{Definition}
\newtheorem{ex}[subsection]{Example}

\theoremstyle{remark}
\newtheorem{rem}[subsection]{Remark}

\numberwithin{equation}{subsection}

\begin{document}
\maketitle

\begin{abstract}
We describe explicitly how certain standard opens of the Hilbert scheme of points are embedded into Grassmannians.  The standard opens of the Hilbert scheme that we consider are given as the intersection of a corresponding basic open affine of the Grassmannian and a closed stratum determined by a Fitting ideal.
\end{abstract}

\section{Introduction}

The main result in this article is an explicit projective embedding of the standard open subschemes of the Hilbert scheme of points in affine space. By taking unions we get an embedding of certain natural subschemes of the Hilbert scheme of points.  Apolarity schemes, commutator ideals and the space of non-degenerate families of $n+1$ points in projective $n$-space are examples and applications of our result.

A standard open subscheme $\Hbeta \subseteq \operatorname{Hilb}^N_X$ of the Hilbert scheme of~$N$ points on an affine scheme $X$ is determined by a sequence $\beta$ of global sections with $|\beta|=N$. The scheme $\Hbeta$ is characterized by parameterizing all closed subschemes of the ambient scheme such that $\beta$ forms a basis for their global sections. Standard open subschemes seem to first have appeared in Haimans article \cite{MR1661369} as a tool in his description of the Hilbert scheme of points in the complex affine plane. The definition of these schemes $\Hbeta$ were formalized and generalized for Hilbert scheme of points on general affine schemes, independently by several authors  \cite{MR2161988}, \cite{MR2324602} 
 and \cite{MR2221028}. 

Considerable interest have been devoted to the schemes $\Hbeta$, and in particular with the Hilbert scheme of points on affine $n$-space and with $\beta$ a sequence of monomials satisfying the order ideal condition, see e.g. \cite{MR2452308}, \cite{MR2116166} 
and \cite{MR2826478}.

However, what have been missing from the discussion of the schemes $\Hbeta$ is a  natural projective embedding. The schemes $\Hbeta$ are defined and realized abstractly, and they glue canonically for different choices of $\beta$. The information is however local, and it has not been clear how to realize these charts and their unions in a projective space. In the present article we give a precise answer to that question. 

Having the sequence of monomials $\beta$ fixed we let $\xbeta{d}$ denote the degree $d$ homogenization, and we view $\xbeta{d}$ as global sections of the degree $d$-forms on projective $n$-space. These global sections $\xbeta{d}$ determine a basic open affine $\Gr{\xbeta{d}}$ of the Grassmannian $\mathbb{G}$ of locally free, rank $N=|\beta|$ quotients of the vector space $V_d$ of $d$-forms on projective $n$-space. We prove the following result.

\begin{summary} Let $\beta$ be an order ideal sequence of monomials in the polynomial ring $A[t_1, \ldots, t_n]$. Let $d$ be an integer such that $d\geq d(\beta)+1$, where $d(\beta)$ is the highest degree of the monomials in $\beta$. Then the standard open subscheme $\Hbeta$ of the Hilbert scheme of $N=|\beta|$ points on affine $n$-space $\A^n_A$, is given as the locally closed subscheme
$$ \Hbeta =\Gr{\xbeta{d}}\cap \operatorname{Fitt}_{N-1}(E)\subseteq \mathbb{G}=\operatorname{Grass}^N(V_d).$$
Here $\operatorname{Fitt}_{N-1}(E)$ is the closed stratum determined by the $(N-1)$'th Fitting ideal of a coherent sheaf $E$ (see Definition \ref{quotient}) on the Grassmannian $\mathbb{G}$.
\end{summary}

Having the length $N$ of the sequence $\beta$ fixed, the sheaf $E$ and the Grassmannian $\mathbb{G}$ depend only on the integer $d$. Thus by fixing $d$, and considering unions, we obtain a description of certain natural subschemes of the Hilbert scheme of points in affine $n$-space, as well as projective $n$-space. We can vary $\beta$ in $A[t_1, \ldots, t_n]$ with $d(\beta)+1\leq d$, and we can in addition vary which hyperplane in projective $n$-space whose complement is our affine $n$-space. The resulting quasi-projective scheme parametrizes closed, length~$N$, subschemes in projective $n$-space, that fiberwise have global sections given by a sequence $\beta$ with $d(\beta)+1\leq d$. We refer to such subschemes as $d$-strongly generated of length $N$ (see Section 3).  These schemes are particular instances of the bounded regularity schemes introduced in \cite{BBRoggero}. For the space of $d$-strongly generated subschemes of length $N$, our result yields an explicit embedding of these parameter schemes. 

We work out a particular example, the 2-strongly generated subschemes of projective $n$-space, of length $n+1$. These subschemes are particularly simple being cut out by quadrics, and we describe how the quasi-projective parameter space is embedded in the Grassmannian of rank $n+1$ quotients of the vector space of degree two forms.

An application of our explicit description is within the theory of apolarity schemes. The scheme $\VPS{n+1}{Z}$ of length $n+1$ subschemes in $\Proj^n$ apolar to the annihilator scheme $Z$ of a smooth quadratic surface was introduced in \cite{RanestadSchreyer} (see Section~\ref{apolarity}). We give a direct argument reproving their result that the space $\VPS{n+1}{Z}$ is closed in the Grassmannian of rank $n+1$ quotients of the vector space of two-forms on projective $n$-space.


In our main result the chosen integer $d$  can typically be lower than the regularity of the subschemes of length $N$. This is in contrast to the assumptions of Gotzmann's persistence theorem \cite{Gotzmann}, a result typically used when considering embeddings of the Hilbert scheme (e.g. \cite{IK}, \cite{BBRoggero}, \cite{ABM}). One advantage with our approach is that we get embeddings of the standard open subschemes $\Hbeta$ in a lower dimensional Grassmannian than what would be possible if using Gotzmann's persistence theorem. Moreover, we can apply our methods to get similar descriptions of Quot schemes, and these generalizations are outlined in the Appendix.

We also establish an unexpected connection between commutators and Fitting ideals. Central in the constructions of the local open schemes $\Hbeta$ are the commutator relations arising from matrices operating on certain vector spaces. We show that the ideal generated by the commutator relations equals the Fitting ideal arising from the graded, global situation. That connection ties together and explains the two different descriptions of the embeddings of the Hilbert scheme of points that appear in \cite{ABM} and in \cite{GFitting} and \cite{QuotinGrass}.

\section{Order ideals and Grassmannians}

\subsection{Order ideals} Let $\beta$ be a finite  collection of monomials in the polynomial ring $A[t_1, \ldots , t_n]$ over a commutative, unitary ring $A$.  We assume that $\beta$ has the property that it contains all its divisors: for any ${\m}\in \beta$, it holds that
\begin{equation}\label{monomial crit} {\m}=t_i^d{\m}' \Rightarrow t_i^{d-1}{\m}' \in \beta,
\end{equation}
for any $i=1, \ldots, n$ and any $d\geq 1$. Such a set $\beta$ is commonly referred to as an {\em order ideal}, see e.g. \cite{MR2161988}.
\begin{lemma} \label{lem:hasbasis}
Let $I\subseteq k[t_1, \ldots, t_n]$ be an ideal in the polynomial ring over a field $k$, such that the quotient ring is a finite dimensional vector space. Then there exists a  sequence $\beta$ of monomials satisfying \eqref{monomial crit}, and such that the images of $\beta$ form a $k$-vector basis for the quotient ring.
\end{lemma}
\begin{proof} This is a consequence of the quotient set having a $k[t_1, \ldots, t_n]$-module structure.
\end{proof}

\subsection{Homogenized monomials} Let $\beta$ be a sequence of monomials in $A[t_1, \ldots, t_n]$, and let
\begin{equation}\label{degree}
 d(\beta)=\operatorname{max}\{\operatorname{deg}({\m}) \mid {\m}\in \beta\}.
\end{equation}
Let $R=A[z, x_1, \ldots, x_n]$ denote the graded polynomial ring in the variables $z, x_1, \ldots, x_n$ over $A$, all having degree one, and let $A[t_1, \ldots, t_n]$ denote the degree zero part of the localization of $R$ with respect to $z$. That is, $t_i=x_i/z$ for any $i=1, \ldots, n$. For any $d\geq d(\beta)$ we denote by
$$ \xbeta{d} =\{ z^d\m \mid \m \in \beta\},$$
considered as a sequence of degree $d$ monomials in $A[z, x_1, \ldots, x_n]$. The cardinality of the sequence $\xbeta{d}$ equals that of $\beta$.  

\subsection{Closed subschemes of projective space}  We say that a closed subscheme ${\Gamma\subseteq\Proj^n_A}$ is \emph{generated in degree $d$} if it is defined by a homogeneous ideal $J\subseteq A[z,x_1,\ldots,x_n]$ generated in degree $d$.

\begin{prop}\label{gen} Let $\Gamma \subseteq \A^n_A$ be a finite closed subscheme, and let $\beta$ be a sequence of monomials in $A[t_1, \ldots, t_n]$ that forms a basis for the global sections of $\Gamma$. We view $\Gamma$ as a closed subscheme in $\Proj^n_A$ by identifying $\A^n_A$ as a hyperplane complement. Then we have that $\Gamma$ is generated in degree $d(\beta)+1$, where $d(\beta)$ is the top degree of the monomials \eqref{degree}.
\end{prop}

\begin{proof} By assumption the monomials in $\beta$ form a basis for the $A$-module $E$ of global sections of $\Gamma$. The algebra structure on $E$ is then completely determined by $t_i\mu$, where $\mu \in \beta$ and $i=1, \ldots, n$. We have that
\begin{equation}\label{mon rel} t_i\mu =\sum_{\lambda \in \beta}a_{i,\lambda}\lambda,
\end{equation}
for some scalars $a_{i, \lambda}$ in $A$. The relation arising from \eqref{mon rel} is a polynomial of degree at most $d(\beta)+1$. Therefore the degree of the generators of the ideal $I\subseteq A[t_1, \ldots, t_n]$ defining $\Gamma$ is bounded by $d(\beta)+1$.
The homogenization of $I$ with respect to $z$ is then generated in degree $d(\beta)+1$, and determines the closed subscheme $\Gamma \subseteq \Proj^n_A$.
\end{proof}

\subsection{Fitting ideals and graded quotients} 
The $A$-module of degree $m$-forms in the standard graded polynomial ring $A[z, x_1, \ldots, x_n]$ we will denote by $V_m$. Thus 
$$\xymatrix{ A[z, x_1, \ldots, x_n]=\bigoplus _{m\geq 0}V_m.}$$
If $J\subseteq A[z,x_1, \ldots, x_n]$ is a graded ideal, we let 
\begin{equation}\label{quotient}
\xymatrix{\Q{J}=\bigoplus_{m\geq0}\Q{J}_m =A[z, x_1, \ldots, x_n]/J}
\end{equation}
denote the graded quotient. Thus, for each $m$  we have that $\Q{J}_m=V_m/J_m$. Finally, for a finitely generated $A$-module $E$ we let $\mathcal{F}itt_N(E)\subseteq A$ denote the $N$'th Fitting ideal of $E$, see \cite[\S3]{Northcott}, \cite[\S20.2]{eisen-comalg}. 

\begin{prop}\label{Gustav}Let $\beta$ be a sequence of monomials in $A[t_1, \ldots, t_n]$ satisfying \eqref{monomial crit}. Fix $d\ge d(\beta)+1$, and let $J\subseteq A[z, x_1, \ldots, x_n]$ be a homogeneous ideal generated in degree $d$. Assume that  $\xbeta{d}$ form a basis for the degree $d$ quotient module 
$ \Q{J}_d = (A[z, x_1, \ldots ,x_n]/J)_d.$
Then we have the equality of Fitting ideals
$$ \mathcal{F}itt_{N-1}(\Q{J}_{d+m})=\mathcal{F}itt_{N-1}(\Q{J}_{d+1}),$$
and we have that $\xbeta{d+m}$ generates $\Q{J}_{d+m}$, for all $m\geq 1$, where $N=|\beta |$. 
\end{prop}

\begin{proof} The result is a reformulation of {\cite[Proposition 2.7]{GFitting}}: as Fitting ideals commute with base change, we can assume that $A$ is a local ring. The assumption on $\beta$, and that $d\geq d(\beta)+1$, implies that $\xbeta{d}$ satisfies Condition~1.2 (op. cit.) and the Proposition 2.7 applies. The statement about $\xbeta{d+m}$ generating the quotient module follows from Corollary 2.5 (op. cit.).
\end{proof}

\subsection{Basic open affines} 

Let $\beta$ be a sequence of monomials in the polynomial ring $A[t_1, \ldots, t_n]$. We fix an integer $d\geq d(\beta)+1$, and let $N=|\beta|$ denote the cardinality of $\beta$.  The sequence
$ \xbeta{d} $
of degree $d$-forms in the graded polynomial ring 
$A[z, x_1, \ldots, x_n]$
determines the basic open affine subscheme
\begin{equation}\label{basic}\Gr{\xbeta{d}} \subseteq \mathbb{G}^N(V_{d})
\end{equation}
of the Grassmannian of locally free, rank $N$, quotients of the $A$-module of homogeneous $d$-forms $V_{d}$. The scheme $\Gr{\xbeta{d}}$ is characterized by being the open set where the images of the elements of $\xbeta{d} $ form a basis for the quotient $V_d\ra E$. See e.\,g.\ \cite[\S8.4]{MR2675155} for details.

\subsection{Fitting strata} Let $\calO_{\mathbb{G}}$ denote the structure sheaf of the Grassmannian $\mathbb{G}^N(V_d)$. On the Grassmannian we have the universal sequence
\begin{equation}\label{univ G} \xymatrix{ 0 \ar[r] & J_d \ar[r] & V_{d,\mathbb{G}} \ar[r] & E_d \ar[r] & 0 }
\end{equation}
of $\calO_{\mathbb{G}}$-modules, where $V_{d,\mathbb{G}}=V_d\otimes_A{\calO}_{\mathbb{G}}$. Then $J_d$ generates a graded ideal sheaf $J=(J_d)\subseteq {\calO}_{\mathbb{G}}[z, x_1, \ldots, x_n]$. 

We let 
\begin{equation}\label{Fitting}
\Fitt_{N-1}(\Q{J}_{d+1})\subseteq \mathbb{G}^{N}(V_d)
\end{equation} denote the closed subscheme defined by the $(N-1)$'th Fitting ideal of the ${\calO}_{\mathbb{G}}$-module $\Q{J}_{d+1} =V_{d+1,\mathbb{G}}/J_{d+1}$.

\begin{thm}\label{beta scheme} Let $A[t_1, \ldots, t_n]$ denote the polynomial ring in the variables $t_1, \ldots, t_n$ over a ring $A$. Let $\beta$ be a sequence of monomials satisfying \eqref{monomial crit}, let $N=|\beta|$ and fix an integer $d\geq d(\beta)+1$.  The functor parameterizing finite closed subschemes $\Gamma\subseteq \A^n_A$ such that $\beta$ forms a basis for the global sections of $\Gamma $ is represented by the scheme 
\[ \xymatrix{ \Hbeta =\Gr{\xbeta{d}}\cap \Fitt_{N-1}(E^{J}_{d+1}) }
\]
where $J=(J_d)\subseteq\calO_{\mathbb{G}}[z,x_1,\ldots,x_n]$ is the ideal generated by the module $J_d$ obtained from the sequence \eqref{univ G}.
The universal family we get by localizing in $z$ and taking out the degree zero part. In particular $\Hbeta$ is a locally closed subscheme of the Grassmannian $\mathbb{G}^N(V_{d})$.
\end{thm}

\begin{proof} By definition we have that $\xbeta{d}$ forms a basis for the universal quotient bundle $E_d$ of the Grassmannian, when restricted to the basic open subscheme $\Gr{\xbeta{d}}$. When we restrict further to the closed subscheme defined by the Fitting ideal $\operatorname{Fitt}_{N-1}(\Q{J}_{d+1})$ we obtain by Proposition \ref{Gustav} that the graded sheaf $\calO[z,x_1, \ldots, x_n]/J$ is free and of rank $N$ in every degree $m\geq d$, where $\calO$ is the structure sheaf on $\Hbeta$. In particular we get a closed subscheme $Z\subseteq \Proj^n_A\times_A \Hbeta$, which is flat, finite with relative rank $N$ over $\Hbeta$ and fiberwise the sequence $\beta$ will form a basis for its global sections. Universal property follows.
\end{proof}

\begin{rem} The scheme $\Hbeta$ described in the Theorem is the basic open subscheme of the Hilbert scheme $\operatorname{Hilb}^N_{\A^n_A}$ of $N$ points on $\A^n_A$, and $\Hbeta$ is parameterizing finite closed subschemes where $\beta$ forms a basis for its global sections. These schemes were introduced in \cite{MR2161988}, \cite{MR2324602} and \cite{MR2221028}, and can be defined for any sequence of elements $\beta$, not only for an order ideal of monomials as we consider. However, the novelty in our description is the projective embedding of these schemes. We do not describe the schemes $\Hbeta$ abstractly, but give these in terms of an explicit embedding into a Grassmannian.
\end{rem}

\begin{rem} The assumption on the integer $d$ in the Theorem is only that $d\geq d(\beta)+1$. Typically that integer $d$ is lower than the regularity for the closed subschemes in $\A^n_A$ of length $N=|\beta|$. In particular
Gotzmann's persistence theorem~\cite{Gotzmann} is not applicable in our situation. 
\end{rem}

\begin{rem}  In the Appendix we show that the above results can be naturally generalized to module quotients. That situation, with Quot schemes replacing Hilbert schemes, is more notationally complex, but otherwise quite similar.
\end{rem}

\begin{ex}\label{smallFitt} We apply our Theorem \ref{beta scheme} to describe ideals in the polynomial ring $A[t_1, t_2]$ having $\beta =\{1, t_1, t_2\}$ as a basis for the quotient ring. We let $d=2=d(\beta)+1$, and then $\xbeta{2}=\{z^2,zx_1, zx_2\}$.  The basic open $\Gr{\xbeta{2}}$ of the Grassmannian of rank 3 quotients of the $A$-module of two-forms $V_2\subseteq A[z,x_1, x_n]$ is an affine 9-space.  Let $B=A[a_i, b_i, c_i]$ where $1\leq i\leq 3,$ denote the polynomial ring in nine variables over $A$. Then $\Spec(B)=\Gr{\xbeta{2}}$, and the universal family over the basic open can be represented by the matrix
$$M =\begin{bmatrix}
1&0&0&a_1&b_1&c_1\\
0&1&0&a_2&b_2&c_2\\
0&0&1&a_3&b_3&c_3
\end{bmatrix}
$$
that fits in the exact sequence of $B$-modules
$$ \xymatrix{0 \ar[r]&  J_2 \ar[r]& V_{2,B}=V_2\otimes_AB \ar[r]^-{M} & E_2 \ar[r]& 0}. 
$$
Here $E_2$ is the free $B$-module with basis $\xbeta{2}$, and $J_2$ is the kernel of $M$. The kernel $J_2$ of $M$  is generated by the three elements 
$${\tiny a=\begin{bmatrix} a_1\\a_2\\a_3\\-1\\0\\0\end{bmatrix},\quad  b=\begin{bmatrix} b_1 \\ b_2\\b_3\\0\\-1\\0\end{bmatrix} }\quad \text{and}\quad  \tiny{c=\begin{bmatrix} c_1\\c_2\\c_3\\0\\0\\-1\end{bmatrix}}.$$ 
The nine elements we obtain by multiplying $a, b$ and $c$ with the linear forms $z,x_1$ and $x_2$ generates the degree three part $J_3$ of the ideal $J=(J_2)$. We order the monomials with the lexicographic ordering, where $z< x_1< x_2$. Then the columns in the $(10\times 9)$-matrix below, represents the generators of~$J_3$.
$$ M'=\tiny{\begin{bmatrix}
a_1 & b_1 & c_1 & 0 & 0 & 0 & 0 & 0 & 0 \\
a_2 & b_2 & c_2 & a_1 & b_1 & c_1 & 0 & 0 & 0 \\
a_3 & b_3 & c_3 & 0 & 0 & 0 & a_1 & b_1 & c_1 \\
-1 & 0 & 0 & a_2 & b_2 & c_2 & 0 & 0 & 0 & \\
0 & -1 & 0 & a_3 & b_3 & c_3 & a_2 & b_2 & c_2 \\
0 & 0 & -1 & 0 & 0 &0 & a_3 & b_3 & c_3 \\
0 & 0 & 0 & -1 & 0& 0& 0&0&0& \\
0&0&0&0&-1&0 &-1&0&0\\ 
0&0&0&0&0&-1& 0&-1&0 \\
0&0&0&0&0&0 & 0&0&-1
\end{bmatrix}}
$$
The cokernel of $M'$ is by definition $\Q{J}_3$, and the Fitting ideal $\mathcal{F}itt_2(\Q{J}_3)$ is the ideal generated by the $(7\times 7)$-minors of $M'$. By computing these minors we get that the ideal cutting out $\Hbeta$ in the affine 9-space, with coordinates $a_1, \ldots, c_3$,  is generated by the five elements
\begin{align*}
&b_2b_3-a_3c_2+b_1,&&b_2^2-a_2c_2+b_3c_2-b_2c_3-c_1,\\&
b_1b_2+b_3c_1-a_1c_2-b_1c_3,&&a_3b_2-a_2b_3+b_3^2-a_3c_3-a_1,\\
&a_2b_1-a_1b_2-b_1b_3+a_3c_1.
\end{align*}
\end{ex}

\begin{ex}\label{ex:fittingex} An efficient way to compute the Fitting ideals described in Theorem~\ref{beta scheme} is available in the 
\emph{Macaulay2}~\cite{M2} program using the package \emph{FiniteFittingIdeals}, available as of version~1.8, and constructed by the second author of the present article. The following is an example of how to use that package.

We will describe  ideals in $A[t_1,t_2]$ where $\beta=\{1, t_1, t_2, t_1^2\}$ forms a basis for the quotient ring. We chose $d=d(\beta)+1=3$. The basic open subscheme $\Gr{\xbeta{3}}$ in the Grassmannian of rank 4 quotients of the $A$-module of 3-forms $V_3\subset A[z,x_1,x_2]$, is an affine 24-space. In order to describe the parameter scheme $\Hbeta$ we need to compute the Fitting ideal $\mathcal{F}itt_3(\Q{J}_4)$, where $J$ is arising from the kernel of the universal family.  Using \emph{Macaulay2}, generators for that ideal can be computed with the commands
{\small \begin{verbatim}
S=ZZ[t_0,t_1,t_2]
G=S[a_1..a_4,b_1..b_4,c_1..c_4,d_1..d_4,e_1..e_4,f_1..f_4]
M=matrix{{1,0,0,0,a_1,b_1,c_1,d_1,e_1,f_1},
         {0,1,0,0,a_2,b_2,c_2,d_2,e_2,f_2},
         {0,0,1,0,a_3,b_3,c_3,d_3,e_3,f_3},
         {0,0,0,1,a_4,b_4,c_4,d_4,e_4,f_4}}
J3=gens ker M
J4=nextDegree(J3,3,S)
co1Fitting(J4) \end{verbatim}
}
The output will be 27 degree two polynomials in the variables $a_1, \ldots, f_4$.
\end{ex}

\section{Strongly generated subschemes} 

\begin{defn} Let $\Gamma \subseteq \A^n_k$ be a closed, finite subscheme with $k$ a field. We say that $\Gamma$ is \emph{$d$-strongly generated} if a sequence of monomials $\beta$ satisfying \eqref{monomial crit} form a basis for the vector space of global sections of $\Gamma$, with 
$$ d\geq d(\beta)+1.$$
A family of closed subschemes $\Gamma \subseteq \A^n_A$, with $A$ a commutative ring, is \emph{$d$-strongly generated} if $\Gamma$ is finite and flat over $A$, and fiberwise $d$-strongly generated.
\end{defn} 

\begin{rem} A family $\Gamma \subseteq \A^n_A$ is $d$-strongly generated, if locally on $\Spec(A)$, we can find a sequence of monomials $\beta$ satisfying~\eqref{monomial crit} that forms a basis for the global sections of $\Gamma$, and where $d\geq d(\beta)+1.$
\end{rem}

\begin{rem} By Proposition \ref{gen} we have that a $d$-strongly generated subscheme $\Gamma \subseteq \A^n_k\subset \Proj^n_k$ is, considered as a subscheme in projective $n$-space, generated in degree $d$.
\end{rem}

\begin{rem} The notion of $d$-strongly generated is stable under base changes, and forms naturally a subfunctor of the Hilbert functor  parametrizing closed subschemes $\Gamma \subseteq \A^n_A$ that are flat and of relative rank $N$ over the base $A$.
\end{rem}

\subsection{Cover} Let $d$ and $N$ be two fixed integers, and let $A[t_1, \ldots, t_n]$ denote the polynomial ring over $A$. Let $\calB$ denote the collection of all (ordered) sequences $\beta$ of monomials in $A[t_1, \ldots, t_n]$ satisfying \eqref{monomial crit}, having cardinality $|\beta|=N$, and top degree $d(\beta)+1\leq d$. Define the scheme
\begin{equation}\label{all beta} \operatorname{Gen}(d,N)=\cup _{\beta \in \calB} \Hbeta \subseteq \mathbb{G}^N(V_{d+1})
\end{equation}
where $\Hbeta$ is the scheme in Theorem \ref{beta scheme}.
\begin{prop}\label{prop:stronglygenerated} The functor that parametrizes closed subschemes $\Gamma \subseteq \A^n_A$ that are $d$-strongly generated of relative rank $N$, is represented by the scheme $\operatorname{Gen}(d,N)$ from \eqref{all beta}. If moreover $d\geq N$, then we have an equality
\[\operatorname{Gen}(d,N)=\operatorname{Hilb}^N_{\A^n_A}.\]
\end{prop}

\begin{proof} Let $\beta$ and $\gamma$ be two sequences of monomials satisfying~\eqref{monomial crit}, with $|\beta|=|\gamma|=N$, and where $d\geq \max\{d(\beta),d(\gamma)\}+1$. 
It follows from the universal property described in Theorem~\ref{beta scheme} that the subscheme
$$ \Hbeta \cap \operatorname{Hilb}^{\gamma} \subseteq \mathbb{G}^N(V_d)$$
parametrizes closed subschemes in $\A^n_A$ where both $\beta$ and $\gamma$ form a basis for its global sections. The first statement then follows. 

If $d\geq N$, then any length $N$-quotient of $k[x_1, \ldots, x_n]$, where $k$ is any field,  has some basis $\beta$ of monomials of the form \eqref{monomial crit} by Lemma~\ref{lem:hasbasis}. Thus the schemes $\Hbeta$ cover $\operatorname{Hilb}^N_{\A^n_A}$, and we have proved the Proposition.
\end{proof}


\begin{rem} When the base field $k$ has infinite cardinality, any finite closed subscheme $\Gamma \subseteq \Proj^n_k$ is contained in a  hyperplane complement $\A^n_k\subset \Proj^n_k$. Thus, by varying different  hyperplanes one obtains, by taking the corresponding unions of \eqref{all beta}, the scheme $\operatorname{Gen}^+(d,N)$ parametrizing finite length~$N$ subschemes $\Gamma \subset \Proj^n_A$ that are flat, and fiberwise $d$-strongly generated. In particular the scheme $\operatorname{Gen}^+(d,N)$ will be a locally closed subscheme of the Grassmannian $\mathbb{G}^N(V_d)$, explicitly described as the unions of the different $\Hbeta$'s intersected with the closed Fitting ideal stratum. The described scheme $\operatorname{Gen}^+(d,N)$ is a special case of the bounded regularity locus introduced and considered in \cite{BBRoggero}. We have not given the details for such a description in general, but will give such a description for a specific example in the next section. 
\end{rem}  

\begin{rem} When the base field $k$ is infinite, and when $d\geq N$, then $d$-strongly generated length $N$ subschemes of $\Proj^n_A$ equals those parametrized by the Hilbert scheme of $N$-points in $\Proj^n_A$. Here $A$ is a $k$-algebra. In that situation the schemes $\operatorname{Gen}^+(d,N)$ we get by varying the hyperplane complement $\A^n_A\subset \Proj^n_A$ will cover the Hilbert scheme $\operatorname{Hilb}^N_{\Proj^n_A}$ of $N$ points in projective $n$-space $\Proj^n_A$. In fact, we have that
$$ \operatorname{Hilb}^N_{\Proj^n_A}=\operatorname{Fitt}_{N-1}(\Q{J}_{d+1})\subseteq \mathbb{G}^N(V_d),$$
where $J$ is the kernel of the universal family on the Grassmannian \eqref{univ G}, see \cite{QuotinGrass}, \cite{GFitting} for details.
\end{rem}

\section{Non-degenerate families of points} 

\subsection{Non-degenerate families} A closed subscheme $\Gamma\subseteq\A_k^n$  is called \emph{non-degenerate} if it is not contained in any hyperplane. A flat family $\Gamma \subseteq \A_A^n$ is non-degenerate if the fiber over any point in $\Spec(A)$ is non-degenerate. We make similar definitions for closed subschemes $\Gamma \subseteq \Proj^n_A$ being non-degenerate.

\begin{prop}\label{affine quadrics} The functor parametrizing closed subschemes $\Gamma \subseteq \A^n_A$ that are flat, finite of relative rank $n+1$, and non-degenerate is represented by the scheme $\Hbeta$ of Theorem \ref{beta scheme}, where
$$ \beta =\{ 1, t_1, \ldots, t_n\}.$$
\end{prop}

\begin{proof} Let $k$ be a field, and let $\Gamma \subseteq \A^n_k$ be a closed subscheme, of rank $n+1$ and non-degenerate. Then we claim that $\beta =\{1, t_1, \ldots, t_n\}$ in the polynomial ring $k[t_1, \ldots, t_n]$ is a basis for the global sections of $\Gamma$. Because, assume that $1, t_1, \ldots, t_{i-1}$ are linearly independent, but that $t_i$ was not a basis element, for some $i>0$. We then have an  equality $t_i=a_0+\sum_{j=1}^{i-1}a_jt_j$. Such an equation would determine a hyperplane containing $\Gamma$, hence contradicting the non-degeneracy assumption. We have then shown that $\beta$ form a basis for the global sections of $\Gamma \subseteq \A^n_A$ over any fiber, hence $\beta$ form a basis for the global sections everywhere. The result then follows from Theorem \ref{beta scheme}.
\end{proof}

\begin{rem}
With $\beta =\{ 1, t_1, \ldots, t_n\}$, the scheme $\Hbeta$ 
equals the scheme $\operatorname{Gen}(2,n)$ of $2$-strongly generated subschemes of length $n+1$ from Proposition~\ref{prop:stronglygenerated}.
\end{rem}

\begin{lemma}\label{hyperplane} Let $\Gamma$ be $n+1$ distinct $k$-rational points in projective $n$-space $\Proj^n_k$ over a field $k$. If $\Gamma$ is non-degenerate there exists a hyperplane $H$ not intersecting $\Gamma$, that is $\Gamma \subseteq \Proj^n_k\setminus H=\A^n_k$.
\end{lemma}

\begin{proof} We will do induction on $n$. The situation with $n=1$ is clear. Let $\Gamma =\{P_1, \ldots, P_n,Q\}$ be $n+1$ points in $\Proj^{n}_k$, and let $H$ be a hyperplane containing $\Gamma'=\{P_1, \ldots, P_n\}$. The non-degeneracy assumption on  $\Gamma $ implies that $Q$ is not in $H$ and that $\Gamma'$, considered as points in $H=\Proj^{n-1}_k$, are non-degenerate. By induction hypothesis there exists a hyperplane $H'\subset \Proj^{n-1}_k$ avoiding the points $\Gamma'$.  The  hyperplane $H'$ together with any point in $R\in \Proj^n_k\setminus \Proj^{n-1}_k$ determines a hyperplane $H'+R$. A dimension count shows that not all hyperplanes $H'+R$ can contain $Q$.
\end{proof}



\begin{rem} The non-degeneracy assumption is necessary. In the projective plane over the field with two elements, any three points lying on a line intersect all the seven lines in the plane (since any two lines intersect and there are only three points on a line). 
\end{rem}

\subsection{Cover II}\label{cover2} Given a set $\Gamma\subseteq\Proj^n_k$ of $n+1$ distinct $k$-rational points, we have, by Lemma~\ref{hyperplane}, that $\Gamma$ is contained in the complement of a hyperplane~$H$. We can choose a finite set $\calB$ of hyperplanes in $\Proj^n_k$ such that any $\Gamma$ is in the complement of one of these hyperplanes. Indeed, this is clear if the field is infinite, and if the field $k$ is finite, then there are only finitely many hyperplanes so the statement is trivial. 

For each linear form $z$ determining a hyperplane in $\calB$  we obtain the polynomial ring $A[t_1, \ldots, t_n]$ after localizing $A[x_0, \ldots, x_n]$ in $z$ and taking out degree zero. In each of these polynomial rings we consider the sequence
$$ \beta_z :=\{1, t_1, \ldots, t_n\}.$$
We define the scheme
\begin{equation}\label{big U}
 \operatorname{Gen}:=\operatorname{Gen}^+(2,n+1) =\cup_{z\in \calB}\operatorname{Hilb}^{\beta_z} \subseteq \mathbb{G}^{n+1}(V_2),
\end{equation}
where $\operatorname{Hilb}^{\beta_z}$ is the scheme of Theorem \ref{beta scheme}.

\begin{prop}\label{prop:nets of quadrics} The functor parametrizing closed subschemes $\Proj^n_A$ that are flat, finite of relative rank $n+1$, and non-degenerate, is represented by the scheme $\operatorname{Gen}$ \eqref{big U}. The universal family is the closed subscheme in $\Proj^n_A\times_A\operatorname{Gen}$ determined by the ideal $(J_2)\subseteq {\calO}[x_0, \ldots, x_n]$ coming from the universal sequence on the Grassmannian $\mathbb{G}^{n+1}(V_2)$ \eqref{univ G}, where $\calO$ is the structure sheaf on $\operatorname{Gen}$.
\end{prop}

\begin{proof} 
Let $z\in \calB$ be a linear form, and let $H\subseteq \Proj^n_A$ denote the hyperplane determined by it. By Proposition \ref{affine quadrics} we have that $\operatorname{Hilb}^{\beta_z}$ is the scheme representing non-degenerate families not having support in the hyperplane $H$. It suffices to check that we have a covering over arbitrary field valued points. Let $\Gamma \subseteq \Proj^n_k$ be a non-degenerate closed subscheme of length $n+1$. It remains to show that there exists a hyperplane $H\in \calB$ that avoids $\Gamma$. We may therefore assume that $\Gamma$ is arranged in the most obtrusive case possible, that is, that $\Gamma$ consists of $n+1$ distinct and $k$-rational points. The result then follows from the construction of $\calB$.
\end{proof}

\section{Commutator relations and Fitting ideals}

We will in this section show a connection between commutator relations and Fitting ideals. 

\subsection{Non-commutative monomials} Let $\hat \beta$ be a sequence of monomials in the free, non-commutative, algebra $A\langle t_1, \ldots, t_n\rangle $ in the variables $t_1, \ldots, t_n$ over a commutative unital ring $A$. We assume that the cardinality ${|\hat \beta |=N}$, and that $1\in \hat \beta$. Let $E_{\hat \beta}$ denote the free $A$-module having as a basis the elements of $\hat \beta$. Assume furthermore that we have $(N\times N)$-matrices $T_1, \ldots ,T_n$, where for any $\m \in \hat \beta$ we have
\begin{equation}\label{action}
T_i(\m) =t_i\m \quad \text{if}\quad t_i\m \in \hat \beta.
\end{equation}
\begin{lemma}\label{matrices} The matrices $T_1 ,\ldots, T_n$ satisfying \eqref{action} determine an $A$-linear map
$$\psi \colon A\langle t_1, \ldots, t_n\rangle \longrightarrow E_{\hat \beta}$$
by sending monomials $t_{n_1}^{a_1}\cdots t_{n_k}^{a_k}$ to $T_{n_1}^{a_1}\cdots T_{n_k}^{a_k}(1)$. The map $\psi$ is surjective and its kernel is a left ideal. And conversely, a left ideal $I\subseteq A\langle t_1, \ldots, t_n\rangle$ where $\hat \beta$ form a basis for the quotient module is given by such matrices $T_1, \ldots, T_n$ satisfying \eqref{action}. 
\end{lemma}

\begin{proof} Let $E=E_{\hat \beta}$. The map $\psi$ is the composition of the evaluation map $\operatorname{ev}\colon \operatorname{End}_A(E) \ra E$ at the element $1\in E$, and the $A$-algebra homomorphism  $T\colon A\langle t_1, \ldots, t_n\rangle \ra \operatorname{End}_A(E)$ determined by the matrices $T_1, \ldots, T_n$. The kernel of the composite map
\[\xymatrix{ A\langle t_1, \ldots, t_n\rangle \ar[r]^-T & \operatorname{End}_A(E) \ar[r]^-{\operatorname{ev}} & E}\]
is a left ideal. 

Conversely, assume $I$ is a left ideal such that $\hat \beta$ form a basis for the quotient module $A\langle t_1, \ldots, t_n\rangle/I =E$. The multiplication action of the free algebra on the quotient $E$ determines the matrices $T_1, \ldots, T_n$ satisfying \eqref{action}. The elements in $\hat \beta$ determine an $A$-module homomorphism ${i_{\hat \beta}\colon E \ra A\langle t_1, \ldots, t_n\rangle}$. The conditions \eqref{action} assure that the composition $\operatorname{ev}\circ T \circ i_{\hat \beta}$ is the identity. 
\end{proof}

\subsection{Commutators} By above, we have that a left ideal $I\subseteq A\langle t_1, \ldots , t_n\rangle $ and monomials $\hat \beta$ that form a basis for the quotient module are equivalent with having matrices $T_1, \ldots, T_n$ satisfying \eqref{action}. In particular such an ideal $I$ and such a sequence $\hat \beta$  determine the commutator ideal in $A$, which is the ideal generated by the entries of the matrices
\begin{equation}\label{comm} [T_i,T_j]=T_iT_j-T_jT_i \quad 1\leq i,j\leq n.
\end{equation}

\subsection{Monomial lift} A lift of a monomial $\m$ in $A[t_1, \ldots, t_n]$ means a monomial word $\hat \m$ in $A\langle t_1, \ldots, t_n\rangle$ such that $\operatorname{c}(\hat \m)=\m$, where 
\begin{equation}\label{canonical} \operatorname{c}\colon A\langle t_1, \ldots, t_n\rangle \ra A[t_1, \ldots, t_n]
\end{equation}
denotes the canonical quotient map to the polynomial ring.
If $\beta$ is a sequence of monomials, then $\hat \beta$ denotes a lifted sequence of monomial words in $A\langle t_1, \ldots, t_n\rangle$. 

\begin{rem}\label{rem:kommutes} Let $\hat \beta$ be a lift of $\beta$, and  assume that we have matrices $T_1, \ldots, T_n$ satisfying \eqref{action}. Then the commutator ideal \eqref{comm} is zero precisely when the matrices commute and the corresponding left ideal $I\subseteq A\langle t_1, \ldots, t_n\rangle$ is such that the canonical map
\[ A\langle t_1, \ldots, t_n\rangle/I \ra A[t_1, \ldots, t_n]/\operatorname{c}(I)\]
is an isomorphism. 
\end{rem}

\subsection{Homogenization} Let $z$ and $x_1, \ldots, x_n$ be variables over $A$. Let $A[z]\langle x_1, \ldots, x_n\rangle$ denote the free $A[z]$-algebra in $x_1, \ldots, x_n$. Note that the variable $z$ commutes with $x_i$ for all $i=1, \ldots, n$. We consider the ring $A[z]\langle x_1, \ldots, x_n\rangle$ as graded where 
$$ \deg(z)=\deg(x_1) =\cdots =\deg(x_n)=1.$$
We can localize $A[z]\langle x_1, \ldots, x_n\rangle$ as an $A[z]$-module. When we localize with respect to $z$ and take out the degree zero part, we obtain the free algebra $A\langle t_1, \ldots, t_n\rangle$ where $t_i=x_i/z$ for $i=1, \ldots, n$.  

\begin{prop}\label{noncommutative} Let $\beta$ be a sequence of monomials in $A[t_1, \ldots, t_n]$ satisfying \eqref{monomial crit}, fix $d\ge d(\beta)+1$, and let $\hat \beta$ be a lift of $\beta$. Let $J'\subseteq A[z]\langle x_1, \ldots, x_n\rangle$ be a graded left ideal, let $J\subseteq A[z, x_1, \ldots, x_n]$ denote its image under the canonical map \eqref{canonical}, and  let $I\subseteq A\langle t_1, \ldots, t_n\rangle$ denote the left ideal we obtain from $J'$ by localization in $z$ and taking out degree zero elements.  We assume the following. 
\begin{enumerate}
\item The graded ideal $J\subseteq A[z,x_1, \ldots, x_n]$ is generated by elements of degree $d$.
\item The sequence $\xbeta{d}$ forms a basis for the $A$-module 
\[\Q{J}_d=(A[z, x_1, \ldots, x_n]/J)_d.\]
\item The sequence $\hat \beta$ forms a basis for the $A$-module $A\langle t_1, \ldots, t_n\rangle/I$. 
\end{enumerate} 
 Then we have the equality of ideals
\[ \mathcal{F}itt_{N-1}(\Q{J}_{d+m})=([T_i,T_j])_{1\leq i,j\leq n} \quad \text{in}\quad A,\]
for all $m\geq 1$, and where $N=|\beta|$.
\end{prop}

\begin{proof} Both the Fitting ideals and commutator ideals commute with base change $A\ra A'$. We also have that the ideals arising from $J'\otimes_AA'$ will also satisfy (1), (2) and (3). Thus to show the proposition we can pass to appropriate quotient rings of $A$. 

We first assume that $[T_i, T_j]=0$ for $1\leq i, j\leq n$. We then have that the left ideal $I\subseteq A\langle t_1, \ldots, t_n\rangle$ is such that
\begin{equation}\label{affine quot} A\langle t_1,\ldots, t_n\rangle /I=A[t_1, \ldots, t_n]/\operatorname{c}(I).
\end{equation}
It follows that the homogeneous left ideal $J'\subseteq A[z]\langle x_1, \ldots, x_n\rangle$ also contains all commutators, that is
\begin{equation}\label{proj quot}
 A[z]\langle x_1, \ldots, x_n\rangle/J'=A[z, x_1, \ldots, x_n]/J.
\end{equation}
One verifies that $J$ is the homogenization of $\operatorname{c}(I)$. The ideal $\operatorname{c}(I)$ determines a closed, finite, subscheme $\Gamma \subseteq \A^n_A$, and the homogeneous ideal $J$ determines $\Gamma$ as a closed subscheme in $\Proj^n_A$. By assumption $\beta$ is an $A$-module basis of the coordinate ring of $\Gamma$, and then we also have that $\xbeta{m}$ is an $A$-module basis for the global sections of $\Gamma \subseteq \Proj^n_A$, for $m\gg 0$, which is the degree $m$ part of the quotient ring \eqref{proj quot}. By the defining properties of Fitting ideals, we then have that  $\mathcal{F}itt_{N-1}(\Q{J}_m)=0$ and that $\mathcal{F}itt_N(\Q{J}_m)=A$ for $m\gg 0$. By Proposition~\ref{gen} we have that $\Gamma$ is generated in degree~$d$, and by assumption $\xbeta{d}$ is a basis for the degree $d$ part of the quotient ring \eqref{proj quot}. It then follows by Proposition~\ref{Gustav} that the Fitting ideals $\mathcal{F}_{N-1}(\Q{J}_{d+m})=0$ for $m\geq 1$. 
We have therefore proved that the Fitting ideals in question are included in the commutator ideal generated by  the entries of the matrices $[T_i,T_j]$ (for $1\leq i, j\leq n$).

To prove the converse, assume that $\mathcal{F}itt_{N-1}(\Q{J}_{d+1})=0$. Proposition \ref{Gustav} gives that $\mathcal{F}itt_{N-1}(\Q{J}_{d+m})=0$, and that $\xbeta{d+m}$ is a basis for the degree~$m$ part of the graded quotient \eqref{proj quot}, for $m\geq 1$. We need to see that the commutators are zero. The closed subscheme $\Gamma \subseteq \Proj^n_A$ determined by the homogeneous ideal $J\subseteq A[z, x_1, \ldots, x_n]$ is then affine, and of relative rank $N=|\beta|$. We then get that $\beta$ is a basis for the coordinate ring of $\Gamma$, which we obtain by localizing \eqref{proj quot} in $z$ and taking out the degree zero part. Let $I'$ be the ideal in $A[t_1, \ldots, t_n]$ corresponding to the closed immersion $\Gamma \subseteq \A^n_A$. As a consequence of exactness of localizations we get a surjection
$$ A\langle t_1, \ldots, t_n\rangle /I \ra A[t_1, \ldots, t_n]/I'.$$
By assumption $\hat \beta$ is a basis for the leftmost module, and we have that $\beta$ is a basis for the module on the right. It follows that the rings are isomorphic. It then follows by Remark~\ref{rem:kommutes} that $[T_i,T_j]=0$ for all $1\leq i,j\leq n$.
\end{proof}

\begin{rem} In \cite{GFitting} and in \cite{ABM} two different descriptions of the embedding of the Hilbert scheme $\operatorname{Hilb}^N_{\Proj^n_A}$ into the Grassmannian $\mathbb{G}^N(V_d)$ were given. The first describes the embedding via Fitting ideals, and the second describes the embedding via commutator ideals. Our result above identifying the Fitting ideal with the commutator ideal connects these two different descriptions of the Hilbert scheme.
\end{rem}

\begin{ex} Let $J'\subseteq A[z]\langle x_1, x_2\rangle$ be the homogeneous left ideal generated by
\begin{align*}
F_1 &=x_1^2-a_1z^2-a_2zx_1-a_3zx_2,\quad & F_2 &=x_2^2-c_1z^2-c_2zx_1-c_3zx_2, \\
F_3 &=x_1x_2-b_1z^2-b_2zx_1-b_3zx_2,\quad &
F_4&=x_2x_1-b_1z^2-b_2zx_1-b_3zx_2,
\end{align*}
where $a_i, b_i, c_i$ ($i=1,2,3$) are elements in $A$. Note that under the canonical map $\operatorname{c}\colon A[z]\langle x_1, x_2\rangle \ra A[z,x_1,x_2]$ we have that $\operatorname{c}(F_3)=\operatorname{c}(F_4)$. Let ${J\subseteq A[z,x_1, x_2]}$ denote the image of $J'$ under the map $\operatorname{c}$. The homogeneous ideal $J=\bigl(\operatorname{c}(F_1),\operatorname{c}(F_2),\operatorname{c}(F_3)\bigr)$ is such that the degree 2 component of the quotient ring $\Q{J}=A[z,x_1,x_2]/J$ is free with basis $\{z^2,zx_1,zx_2\}$. 

On the other hand, the left ideal $I\subseteq A\langle t_1, t_2\rangle$, obtained by localizing $J'$ in $z$ and taking out the degree zero part, is generated by
\begin{align*}
f_1 &=t_1^2-a_1-a_2t_1-a_3t_2,\quad & f_2 &=t_2^2-c_1-c_2t_1-c_3t_2, \\
f_3 &=t_1t_2-b_1-b_2t_1-b_3t_2,\quad &
f_4&=t_2t_1-b_1-b_2t_1-b_3t_2.
\end{align*}
The quotient $A\langle t_1,t_2\rangle /I=E$ is the free $A$-module with basis $ \{1, t_1, t_2\}$. We are therefore in a situation where  Proposition \ref{noncommutative} applies. We have $\beta =\{1, t_1, t_2\}=\hat \beta$, and use degree $d=2=d(\beta)+1$. The left ideal $J'=(F_1, F_2, F_3 ,F_4) \subset A[z]\langle x_1,x_2\rangle$ is such that the assumptions (1), (2) and (3) are satisfied.  Now we want to describe the commutator ideal and the Fitting ideal arising in this particular situation.

The matrices, corresponding to the action of $A\langle t_1, t_2\rangle$ on the quotient module $E$ are
$$ T_1=\begin{bmatrix}
0&a_1&b_1\\1&a_2&b_2\\0&a_3&b_3
\end{bmatrix}
\quad\text{and}\quad
T_2=\begin{bmatrix}
0&b_1&c_1\\0&b_2&c_2\\1&b_3&c_3
\end{bmatrix}. $$

The commutator ideal in $A$ is generated by the coefficients of  $T_1T_2-T_2T_1$, and that matrix is 
$$\begin{bmatrix}
0&a_1b_2-a_2b_1+b_1b_3-a_3c_1& a_1c_2-b_1b_2-b_3c_1+b_1c_3\\
0&b_2b_3-a_3c_2+b_1                 &a_2c_2-b_2^2-b_3c_2+b_2c_3+c_1\\
0&a_3b_2-a_2b_3+b_3^2-a_3c_3-a_1&a_3c_2-b_2b_3-b_1
\end{bmatrix}.$$

The commutator ideal is generated by five elements as the two lower diagonal elements are equal up to a sign. On the other hand we want to compute the 2nd Fitting ideals of the graded components of $\Q{J}=A[z,x_1,x_2]/J$, in degrees $\geq 3$. The Fitting ideal of interest is  $\Fitt_{2}(\Q{J}_3)$, the one arising from the degree 3 component. That ideal we computed in Example \ref{smallFitt}, and it is readily verified that those five generators listed  are, up to signs, the five generators of the commutator ideal.  In general the number of generators of the Fitting ideal that we obtain by taking all the prescribed minors,  will be far more than the number of generators of the commutator ideal. 
\end{ex}

\begin{rem} A technical comment. If one starts with a free $A$-module $E$ with basis $\{1, t_1, t_2\}$, and require $E$ to be a quotient module of $A\langle t_1, t_2\rangle$, one obtains the matrices
$$ T_1=\begin{bmatrix}
0&a_1&b_1\\1&a_2&b_2\\0&a_3&b_3
\end{bmatrix}
\quad\text{and}\quad
T_2=\begin{bmatrix}
0&b_1'&c_1\\0&b_2'&c_2\\1&b_3'&c_3
\end{bmatrix}. $$
The second column in $T_2$ is here not the same as the third column of $T_1$, as was the case in the above example. Backtracking gives now a graded left ideal $J'=(F_1, F_2, F_3,F_4)\subseteq A[z]\langle x_1, x_2\rangle$ generated by four elements where $F_1, F_2$ and $F_3$ are as above, but where 
$$ F_4=x_2x_1-b_1'z^2-b_2'zx_1-b_3'zx_2.$$
In general we will not have $\operatorname{c}(F_3)$ equal to $\operatorname{c}(F_4)$ via the canonical map $\operatorname{c}\colon A[z]\langle x_1, x_2\rangle \ra A[z,x_1,x_2]$. And then we will not have that the quotient module $E=A[z,x_1,x_2]/J=\bigl(\operatorname{c}(F_1),\operatorname{c}(F_2),\operatorname{c}(F_3),\operatorname{c}(F_3)\bigr)$ will be free in degree 2. In particular we see that the condition (3) of Proposition \ref{noncommutative} does not imply condition (2). 
\end{rem}

\section{Apolarity schemes}

\subsection{Cones} Let $A[z,x_1, \ldots, x_n]$ be the homogeneous coordinate ring of projective $n$-space $\Proj^n_A$ over $A$. Let $\Gamma \subseteq \Proj^n_A$ be a closed subscheme, and let $\calI$ denote its corresponding ideal sheaf. The affine cone is the closed subscheme $C_{\Gamma}\subseteq \A^{n+1}_A$ defined by the ideal
$$ \xymatrix{\bigoplus _{m\geq 0}H^0(\calI (m))\subseteq A[z, x_1, \ldots, x_n].}$$
\subsection{Apolarity}\label{apolarity} Let $Z\subseteq \A^{n+1}_A$ be a closed subscheme. Inspired by \cite{RanestadSchreyer} we say that a finite subscheme $\Gamma \subseteq \Proj^n_A$ is {\em apolar} to $Z$ if the affine cone of $\Gamma$ contains $Z$ as a subscheme, that is
$$ Z\subseteq C_{\Gamma} \quad \text{in}\quad  \A^{n+1}_A.$$
We let $\VPS{N}{Z}(A)$ denote the set of closed subschemes $\Gamma \subseteq \Proj^n_A$ that are finite, flat, and of relative rank $N$, and such that $\Gamma$ is apolar to $Z$. Then $\VPS{N}{Z}$ naturally becomes a subfunctor of the Hilbert functor ${\operatorname{Hilb}^N_{\Proj^n_A}}$ parametrizing closed subschemes in $\Proj^n_A$ that are flat, finite, and of relative rank $N$.

\begin{rem} We do not impose the reduced structure on $\VPS{N}{Z}$, a condition that is assumed in \cite{RanestadSchreyer}, \cite{MR3084557}. Note also that our definition of apolarity is slightly more general than the one given in \cite{RanestadSchreyer}.
\end{rem}

\begin{rem}
In \cite{RanestadSchreyer}, the authors define $\textrm{VPS}$ as the {\bf v}ariety of a{\bf p}olar {\bf s}chemes and $\textrm{VSP}$ as the {\bf v}ariety of {\bf s}ums of {\bf p}owers. In \cite{MR3084557} the notation appears to have changed with $\textrm{VPS}$ for the variety of sums of powers and $\textrm{VAPS}$ for the variety of apolar schemes. We use the notation from the first paper. 
\end{rem}


\begin{prop}\label{VPS} Let $I\subseteq A[z, x_1, \ldots, x_n]$ be an ideal generated by homogeneous elements, and let $Z\subseteq \A^{n+1}_A$ denote the corresponding, affine,  closed subscheme. Assume that $Z$ is finite and flat over $A$. Then $\VPS{N}{Z}$ is a closed subscheme of $\operatorname{Hilb^N_{\Proj^n_A}}$.
\end{prop}

\begin{proof} Let $\calI$ denote the ideal sheaf of a closed subscheme $\Gamma \subseteq \Proj^n_A$, where $\Gamma$ is flat and of finite rank $N$ over $A$. The condition of $\Gamma$ being apolar to $Z$  is that the ideal $\bigoplus_{m\geq 0}H^0(\calI(m))$ in $A[z, x_1, \ldots, x_n]$ is included in $I$. Such an inclusion is equivalent with the natural map 
\begin{equation}\label{graded map}
 \xymatrix{ \bigoplus_{m\geq 0}H^0(\calI(m)) \ar[r] & A[z, x_1, \ldots, x_n]/I=E }
\end{equation}
being the zero map. By assumption $I$ is graded, so the quotient module $E$ is also graded, $E=\bigoplus_{m\ge0} E_m$. Moreover,  the vanishing of \eqref{graded map} is equivalent with the vanishing of $
\xymatrix{ H^0(\calI(m)) \ar[r] & E_m }
$
in every degree $m\geq 0$. As $Z$ is assumed to be flat and finite, we have that $E$ is locally free as an $A$-module. Then $E_m$ is locally free in each degree $m\geq 0$. Thus, each degree $m$ component of the map \eqref{graded map} is, due to the hom-tensor adjunction, equivalent with the induced map
\begin{equation}\label{dual map} \xymatrix{ H^0(\calI(m)) \otimes_A E_m^{*} \ar[r] & A, }
\end{equation}
where $E_m^*=\operatorname{Hom}_A(E_m,A)$ is the dual module of $E_m$. The image of the map \eqref{dual map} is an ideal that corresponds to the closed subscheme on $\Spec(A)$ where \eqref{dual map} vanishes. Taking the intersection of these closed subschemes for all $m$ determines the closed criterion on $\Spec(A)$ where \eqref{graded map} vanishes. We then have that $\VPS{N}{Z}$ is closed in the Hilbert scheme.   
\end{proof}

\subsection{Differential operators}\label{Fperp} Let $F=F(z, x_1, \ldots, x_n)$ be a homogeneous, quadratic polynomial in $A[z, x_1, \ldots, x_n]$. We have the differential operator $\partial F$ on $A[z, x_1, \ldots, x_n]$ sending an element $f=f(z, x_1, \ldots, x_n)$ to 
$$ \partial{F}(f) =F\biggl(\frac{\partial  }{\partial z}, \frac{\partial }{\partial x_1}, \ldots, \frac{\partial }{\partial x_n}\biggr)(f).$$
If $f$ is of degree two, the degree of $F$, the element ${\partial F}(f)$ is a scalar.  Let $H^0(\calO(2))$ denote the $A$-module of two-forms. We obtain an $A$-linear map 
$$ {\partial{F}}\colon H^0(\calO(2)) \ra A,$$
by differentiating as dictated by $F$. The kernel of ${\partial{F}}$ is a submodule of $H^0(\calO(2))$, and generates an ideal $F^{\perp}$ in the polynomial ring $A[z, x_1, \ldots, x_n]$.

\begin{prop}\label{apolar} Let $F=F(z, x_1, \ldots, x_n)$ in $k[z, x_1, \ldots, x_n]$ be a quadratic polynomial defining a smooth hypersurface in $\Proj^n_k$, with $k$ a field of zero characteristic. Let ${Z(F)\subseteq \A^{n+1}_k}$ denote the closed subscheme defined by the ideal $F^{\perp}$ \eqref{Fperp}. We then have that $\VPS{n+1}{Z(F)}$, the space of $n+1$ points in $\Proj^n_k$ apolar to $Z(F)$,  is a closed subscheme of the Grassmannian of rank $n+1$ quotients of the $k$-vector space of two-forms $H^0(\calO(2))$.
\end{prop}

\begin{proof} The scheme $Z(F)\subseteq \A_k^{n+1}$ defined by the ideal $F^{\perp}\subseteq k[z,x_1, \ldots, x_n]$ is finite, and flat since the base is a field. By definition the ideal $F^{\perp}$ is generated by homogeneous elements of degree two, and it follows by Proposition~\ref{VPS} that $\VPS{n+1}{Z(F)}$ is representable by a scheme. Let $\calI$ denote the ideal sheaf of a flat family $\Gamma \subseteq \Proj^n_A$ of relative rank $n+1$ over a $k$-algebra~$A$. Assume that the $A$-valued point of the Hilbert scheme $\operatorname{Hilb}^{n+1}_{\Proj^n}$ is an $A$-valued point of $\VPS{n+1}{Z(F)}$. Then we have an inclusion of ideals
$$\xymatrix{ \bigoplus_{m\geq 0}H^0(\calI(m))\subseteq F^{\perp} \quad \text{in} \quad A[z,x_1, \ldots, x_n]}.$$
In particular we have that $H^0(\calI(0))=H^0(\calI(1))=0$. Moreover, since the quadratic hyperplane in $\Proj^n$ determined by $F$ is non-singular, it follows that $H^0(\calI(2))$ contains no ``fake'' hyperplanes either. That is, the ideal does not contain degree 2 elements of the form  $lz, lx_1,\ldots ,  lx_n$, for some linear form $l$ in $k[z, x_1, \ldots, x_n]$.  That means that for any point in $\Spec(A)$ the fiber of the universal family is non-degenerate, hence the whole family over $\Spec(A)$ is a non-degenerate family. Thus, $\Spec(A)$ is an $A$-valued point of the scheme~$\operatorname{Gen}$ from Proposition~\ref{prop:nets of quadrics}. The scheme $\operatorname{Gen}$ is a locally closed subscheme of the Hilbert scheme $\mathrm{Hilb}^{n+1}_{\Proj^n}$ and Proposition~\ref{prop:nets of quadrics}  tells us that $\operatorname{Gen}$ is embedded in the Grassmannian $\mathbb{G}^{n+1}(V_2)$ of rank $n+1$ quotients of the $A$-module of two-forms. As $\VPS{n+1}{Z(F)}$ is closed in the Hilbert scheme, but also closed in the space $\operatorname{Gen}$, it follows that it is closed in the Grassmannian.
\end{proof}

\begin{rem} Proposition \ref{apolar} is stated in \cite[Corollary 2.2]{MR3084557}. Their result is, in parts, based on the computation of the graded Betti numbers of non-degenerate length $n+1$ subschemes in $\Proj^n_k$. Our proof is perhaps more transparent. In any case our result concerns an embedding of the scheme $\VPS{n+1}{Z(F)}$ with its possibly non-reduced structure.
\end{rem} 
\appendix
\section{Strongly generated quotient sheaves}

We restate several of our previous results in a more general setting with modules and Quot schemes replacing ideals and Hilbert schemes. Standard basic opens of the Quot scheme were introduced in  \cite{MR2346502}. 

\subsection{Monomial bases} Let $S=A[t_1,\ldots,t_n]$ be a polynomial ring over a commutative unitary ring $A$ and let $F=\bigoplus_{i=1}^m Se_i$ be a free $S$-module. By a monomial in $F$, we mean an element of the form ${\m}e_i$, with ${\m}$ a monomial in the polynomial ring $S$. Let $\beta$ be a finite  collection of monomials in $F$.  We assume that $\beta$ has the property that it contains all its divisors: for any ${\m}e_i\in \beta$, it holds that
\begin{equation}\label{monomial crit2} {\m}e_i=t_i^d{\m}'e_i \Rightarrow t_i^{d-1}{\m}'e_i \in \beta,
\end{equation}
for any $i=1, \ldots, n$ and any $d\geq 1$.
\begin{lemma} Let $R\subseteq F=\bigoplus_{i=1}^m k[t_1, \ldots, t_n]$ be a submodule where $k$ is a field, such that the quotient $F/R$ is a finite dimensional $k$-vector space. Then there exists a  sequence $\beta$ of monomials satisfying \eqref{monomial crit2} such that the images of the elements of $\beta$ form a $k$-vector basis for the quotient module.
\end{lemma}
\begin{proof} This is a consequence of $F/R$ being a $k[t_1,\ldots,t_n]$-module.
\end{proof}


\begin{prop}\label{gen2}  Let $S=A[t_1,\ldots,t_n]$ be a polynomial ring over a commutative unitary ring $A$ and let $F=\bigoplus_{i=1}^m Se_i$ be a free $S$-module. Let $R\subseteq F$ be a submodule such that $F/R$ is free of finite rank as an $A$-module, and let $\beta$ be a sequence of monomials in $F$, satisfying \eqref{monomial crit2}, that form an $A$-basis for $F/R$. Then we have that $R$ is generated in degree $d(\beta)+1$, where $d(\beta)$ is the top degree of the monomials. 
\end{prop}

\begin{proof} The proof is similar to the proof of Proposition~\ref{gen}. 
\end{proof}

\subsection{Grassmannian} 
Let $A[z, x_1, \ldots, x_n]=\bigoplus_{d\ge0}V_d$ denote the standard graded polynomial ring, and let $V_{d}^m$ denote $m$ copies of the degree $d$-forms, that is $V_d^m=\bigoplus_{i=1}^mV_d$. On the Grassmannian $\mathbb{G}^N(V_{d}^m)$ of locally free rank~$N$ quotients of $V_{d}^m$ we have the universal sequence 
$$ \xymatrix{ 0\ar[r]& J_d \ar[r] & V^m_{d,\mathbb{G}} \ar[r]& E_d \ar [r] & 0.}$$
Consider the sheaf $\mathcal{S}=\calO_{\mathbb{G}}[z,x_1,\ldots,x_n]$ on the Grassmannian. The kernel~$J_d$ generates a graded submodule $J=J_d\cdot \mathcal{S}$ of the free $\mathcal{S}$-module  $\mathcal{F}=\bigoplus_{i=1}^m\mathcal{S}e_i$, and we let $\Q{J}$ denote the graded quotient module.

\begin{thm}\label{beta scheme2} Let $S=A[t_1, \ldots, t_n]$ denote the polynomial ring in the variables $t_1, \ldots, t_n$ over a ring $A$ and let $F=\bigoplus_{i=1}^m Se_i$. Let $\beta$ be a sequence of monomials satisfying \eqref{monomial crit2}, and fix an integer $d\geq d(\beta)+1$.  The functor parameterizing $S$-module quotients $F/R$ such that $\beta$ is an $A$-module  basis for the quotient, is represented by the scheme 
$$\xymatrix{ \Gr{\xbeta{d}}\cap \Fitt_{N-1}(\Q{J}_{d+1}),}$$
where $N=|\beta|$. Here $\Gr{\xbeta{d}}$ is the basic open in the Grassmannian $\mathbb{G}^N(V_{d}^m)$. 
\end{thm}

\begin{proof} The proof is similar to the proof of Theorem \ref{beta scheme}. Instead of using Proposition~\ref{Gustav}, one uses the original statement \cite[Proposition 2.7]{GFitting} that is valid for modules.
\end{proof}

\subsection{Quot scheme} Let $k$ be a field, and $S=k[t_1, \ldots, t_n]$. We say that an $S$-submodule $J\subseteq F=\bigoplus_{i=1}^mSe_i$ is $d$-strongly generated if there is a sequence of monomials $\beta$ in $F$ satisfying \eqref{monomial crit2} that form a vector space basis for the quotient, and where $d\geq d(\beta)+1$.

Fix $S=A[t_1, \ldots, t_n]$ and an $S$-module $F$. The Quot functor $\operatorname{Quot}^N_{F/S/A}$ of Grothendieck parametrizes $S$-module quotients of $F$ that are flat, finite of relative rank $N$ over the base $A$ \cite{MR1611822}. We consider the subfunctor parametrizing $S$-module quotients of $F$ that are flat, finite of relative rank $N$ and fiberwise $d$-strongly generated.
\subsection{Cover III} Let $d$ and $N$ be fixed integers, and let $S=A[t_1, \ldots, t_n]$. Let $\calB$ denote the collection of all sequences of monomials $\beta$ in $F=\bigoplus_{i=1}^mSe_i$ satisfying \eqref{monomial crit2}, where $|\beta|=N$ and where $d\geq d(\beta)+1$. We define the scheme
\begin{equation}\label{quot}
\operatorname{Gen}^m(d,N) =\cup_{\beta \in \calB}  \Gr{\xbeta{d}}\cap \Fitt_{N-1}(\Q{J}_{d+1}),
\end{equation}
where we use the notation from Theorem \ref{beta scheme2}.

\begin{cor} The functor that parametrizes $S=A[t_1, \ldots, t_n]$-module quotients of $F=\bigoplus_{i=1}^m Se_i$ that are $d$-strongly generated of relative rank $N$, is represented by the scheme $\operatorname{Gen}^m(d,N)$ \eqref{quot}. Moreover, if $d\geq N$, then we have the equality
\[ \operatorname{Gen}^m(d,N)=\operatorname{Quot}^N_{F/S/A}.\]
\end{cor}
\begin{proof} The proof is similar to the proof of Proposition~\ref{prop:stronglygenerated}.
\end{proof}
\bibliography{references}{}
\bibliographystyle{amsalpha}
 
\end{document}